\documentclass{amsart}
\usepackage{amsfonts,amssymb,amscd,amsmath,enumerate,verbatim,calc}
\usepackage[all]{xy}

\newcommand{\CM}{Cohen-Macaulay}

\newcommand{\ff}{\text{if and only if}}
\newcommand{\wrt}{with respect to}

\newcommand{\X}{\mathbf{X} }
\newcommand{\C}{\mathbf{C}^\bullet }
\newcommand{\D}{\mathbf{D}_\bullet }
\newcommand{\K}{\mathbb{K}_\bullet }
\newcommand{\Y}{\mathbf{Y}_\bullet }

\newcommand{\m}{\mathfrak{m} }

\newcommand{\q}{\mathfrak{q} }
\newcommand{\af}{\mathfrak{a} }
\newcommand{\bbf}{\mathfrak{b}}
\newcommand{\R}{\mathcal{R}(I) }
\newcommand{\Rc}{\mathcal{R} }
\newcommand{\eR}{\widehat{\mathcal{R}}(I) }

\newcommand{\Z}{\mathbb{Z} }
\newcommand{\Vc}{\mathcal{V} }

\newcommand{\pr}{\prime}
\newcommand{\rt}{\rightarrow}
\newcommand{\xar}{\longrightarrow}
\newcommand{\ov}{\overline}

\newcommand{\bx}{\mathbf{x}}

\newcommand{\bu}{\mathbf{u}}

\newcommand{\core}{\operatorname{core}}

\newcommand{\depth}{\operatorname{depth}}

\newcommand{\ann}{\operatorname{ann}}

\newcommand{\Tot}{\operatorname{Tot}}

\theoremstyle{plain}

\newtheorem{theorem}{Theorem}[section]
\newtheorem{corollary}[theorem]{Corollary}
\newtheorem{lemma}[theorem]{Lemma}
\newtheorem{proposition}[theorem]{Proposition}
\newtheorem{question}[theorem]{Question}

\theoremstyle{definition}

\newtheorem{s}[theorem]{}
\newtheorem{remark}[theorem]{Remark}

\theoremstyle{remark}

\numberwithin{equation}{theorem}

\begin{document}

\title{On the intersection of Annihilator of the Valabrega-Valla module }
 \author{Tony~J.~Puthenpurakal}
\date{\today}
\address{Department of Mathematics, Indian Institute of Technology Bombay, Powai, Mumbai 400 076}

\email{tputhen@math.iitb.ac.in}

\thanks{While writing this  paper the  author was a visitor at University of Kentucky, under a fellowship from DST, India. The author thanks both DST and UK for its support}


\begin{abstract}
Let $(A,\m)$ be a \CM \ local ring with an infinite residue field and let $I$ be an $\m$-primary ideal. 
Let $\bx  = x_1,\ldots,x_r$ be a $A$-superficial sequence \wrt \ $I$. Set
$$\Vc_I(\bx) = \bigoplus_{n\geq 1} \frac{I^{n+1}\cap (\bx) }{\bx I^n}. $$
A consequence of a theorem due to Valabrega and Valla is that $\Vc_I(\bx) = 0$ \ff \ the initial forms $x_1^*,\ldots,x_r^*$ is a $G_I(A)$ regular sequence. Furthermore this holds if and only if  $\depth G_I(A) \geq r$. We show that if $\depth G_I(A) < r$  then
\[
\af_r(I)= \bigcap_{\substack{\text{$\bx  = x_1,\ldots,x_r$ is a } \\ \text{ $A$-superficial sequence w.r.t $I$}}} \ann_A \Vc_I(\bx) \quad \ \text{is} \ \m\text{-primary}.
\]
Suprisingly we also prove that under the same hypotheses,
\[
\bigcap_{n\geq 1} \af_r(I^n) \quad \ \text{is also} \ \m\text{-primary}.
\]
\end{abstract}

\maketitle

\section*{Introduction}
Let $(A,\m)$ be a Noetherian local ring with an infinite residue field.
The notion of minimal reduction of an ideal $I$ in  $A$ was discovered more than fifty years ago by Northcott and Rees; \cite{Northcott-Rees}. It plays an essential role in the study of blow-up algebra's. Nevertheless minimal reductions are highly non-unique.  The intersection of all minimal reductions is named as \textit{core} of $I$ and denoted by
$\core(I)$. This was introduced by Rees and Sally in \cite{Rees-Sally}. It has been extensively investigated in

 \cite{CPU-Ann},\cite{CPU-TAMS} and \cite{Hun-Trung}. When $A$ is \CM \ and $I$ is $\m$-primary; Rees and Sally proved that $\core(I)$ is again $\m$-primary and so is a finite intersection.
In this paper we study a different intersection of ideals.

 Let $(A,\m)$ be a \CM \ local ring of dimension $d$ with an infinite residue field and let $I$ be an $\m$-primary ideal. Let $\bx  = x_1,\ldots,x_r$ be a $A$-superficial sequence \wrt \ $I$. Set
$$\Vc_I(\bx) = \bigoplus_{n\geq 1} \frac{I^{n+1}\cap (\bx) }{\bx I^n}. $$
We call $\Vc_I(\bx)$ the  Valabrega-Valla module of $I$ \wrt \ $\bx$. 
 A consequence of a theorem due to Valabrega and Valla, \cite[2.3]{VV} is that $\Vc_I(\bx) = 0$ \ff \  the initial forms $x_1^*,\ldots,x_r^*$ is a $G_I(A)$ regular sequence. Furthermore this holds if and only if  $\depth G_I(A) \geq r$, see \cite[2.1]{HM}. In general notice each $\Vc_I(\bx)$ has finite length and so $\ann_A \Vc_I(\bx)$ is $\m$-primary. We prove, see Theorem \ref{main},
 that
 \[
 \af_r(I) = \bigcap_{\substack{\text{$\bx  = x_1,\ldots,x_r$ is a } \\ \text{ $A$-superficial sequence w.r.t $I$}}} \ann_A \Vc_I(\bx) \quad \ \text{is} \ \m\text{-primary}.
\]

Our intersection of ideals  is in some sense analogous to that of core of $I$; since
notice that
$$ \core(I) = \bigcap_{\substack{\text{$J$ minimal } \\ \text{ reduction of $I$}}} \ann_A  \frac{A}{J}. $$
Nevertheless they are two different invariants of $I$. Furthermore our techniques are totally different from that in the papers listed above.

By a result of Elias $\depth G_{I^n}(A)$ is constant for all $n \gg 0$, see \cite[2.2]{Elias}.
Since $\core(I) \subseteq I$ we have $\bigcap_{n\geq 1} \core(I^n) = 0$. Suprisingly, see Theorem \ref{main-2}, we have 
that if $\depth G_I(A) < r$ then
\[
\bigcap_{n\geq 1} \af_r(I^n) \quad \ \text{is} \ \m\text{-primary}.
\]

We now assume $A$ is also complete. Let $\R = \bigoplus_{n\geq 0}I^n$ be the Rees algebra of $I$. Set $L = L^I(A)= \bigoplus_{n\geq 0}A/I^{n+1}$. It can be shown easily that $L$ is a $\R$-module.
Of course $L$ is not finitely generated as a $\R$-module. Nevertheless we prove that its local cohomology modules $H^i_{\R_+}(L)$ are *-Artinian for $i = 0,\ldots, d-1$; see Theorem \ref{locArt}. Recall a graded $\R$-module $N$ is said to be $*$-Artininan if it satisfies d.c.c on its graded submodules. Set $\bbf_i(I) = \ann_{\R} H^i_{\R_+}(L)$  for $i = 0,\ldots, d-1$ and set $\q_i(I) = \bbf_i(I) \cap A$. Since $H^i_{\R_+}(L)$ is *-Artinian; it is not
 so difficult to show that $\q_i$ is $\m$-primary (or equal to $A$); see Corollary \ref{loc-CA}.

In Theorem \ref{man}  we prove that
\[
\af_r(I) \supseteq   \q_0(I)\q_1(I)\cdots\q_{r-1}(I).
\]

Next note that $L^I(A)(-1)$ behaves well  with respect to the Veronese functor. Clearly
\[
\left(L^I(A)(-1)\right)^{<l>}  = L^{I^l}(A)(-1) \quad \text{for each} \ l \geq 1.
\]
Also  local cohomolgy commutes with the Veronese functor. As a consequence we have
\[
\q_i(I^l) \supseteq \q_i(I) \quad \text{for each} \ l \geq 1 \ \text{and} \ i = 0,1,\ldots,r-1.
\]
It follows that 
\[
\bigcap_{n \geq 1}\af_r(I^n) \supseteq   \q_0(I)\q_1(I)\cdots\q_{r-1}(I).
\]

The $\R$-module $L^I(A)$ is not finitely generated $\R$-module. However it is quasi-finite $\R$-module, see section 1.5.
 Quasi-finite module were introduced in
\cite[page 10]{HPV}.   Surprisingly we were able to prove that if $E$ is a quasi-finite $\R$-module and has  a filter-regular sequence of length $s$ then
the local cohomology modules $H^{i}_{\R_+}(E)$ are all *-Artinian for $i = 0,\ldots, s-1$.

We also study the Koszul homology of a quasi-finite module \wrt \ a filter regular sequence. We then use a spectral sequence, first used
by P. Roberts \cite[Theorem 1]{Roberts}, to relate cohomological annihilators with that of annihilators of the Koszul complex. We however have to very careful in our proof since we are dealing with infinitely generated modules.

We now describe in brief the contents of this paper. In section 1 we introduce notation
and discuss a few preliminary facts that we need. In section 2 we study a few basic properties of $L^I(M)$. In section 3 we prove some properties of Koszul homology of quasi-finite modules  with respect to filter-regular sequence. We also compute $H_1(\bu, L^I(M))$ where  $\bu = x_1t,\ldots,x_rt \in \R_1$ is a $L^I(M)$-filter regular sequence. In section 4 we study local cohomology of quasi-finite modules $E$ \\ with
            $\ell(E_n)$ finite for all $n \in \Z$. In section 5 we prove that $\af_r(I)$ is $\m$-primary (or $A$).
In section 6 we show that  $\bigcap_{n\geq 1} \af_r(I^n)$ is $\m$-primary (or $A$).

\section{Notation and Preliminaries}
Throughout we assume that $(A,\m)$ is a Noetherian local ring with an infinite residue field $k =A/\m$. Let $M$ be a finitely generated \ $A$-module of dimension $r$ and let $I$ be an ideal of definition for $M$; i.e, $\ell(M/IM)$ is finite. Here $\ell(-)$ denotes length. For undefined terms see
\cite{BH}, especially sections 4.5  and 4.6.

\s Assume $r = \dim M \geq 1$. Let $x \in I\setminus I^2$.  We say $x$ is $M$-\emph{superficial} \wrt \ $I$  if for some  $c\geq 1$ we have
$(I^{n+1}M \colon x )\cap I^cM = I^nM$
 for all $n \gg 0$. If $\depth M > 0$ then using the Artin-Rees Lemma one can  prove that $(I^{n+1}M \colon x ) = I^nM$
 for all $n \gg 0$.

  Superficial sequences can be defined as usual. Since $k$ is infinite $M$-superficial sequences of length $ r = \dim M$ exists.

\s Let $\bx = x_1,\ldots, x_r$ be a $M$-superficial sequence  \wrt \ $I$. The Valabrega-Valla module of $I$ \wrt \ $M$ and $\bx$ is
\[
\Vc_I(\bx, M) = \bigoplus_{n\geq 1} \frac{I^{n+1} M\cap \bx M}{\bx I^nM}.
\]
We consider it as a $A$-module.
Set $\Vc_I(\bx) = \Vc_I(\bx, A)$.

\s Let $\eR = \bigoplus_{n \in \Z}I^nt^n$ denote the extended Rees-algebra of $A$ \wrt \ $I$. Here $I^n = A$ for $n \leq 0$. We consider it as a subring of $A[t,t^{-1}]$.
Let $\R = \bigoplus_{n \geq 0 }I^nt^n$ denote the  Rees-algebra of $A$ \wrt \ $I$. We consider it as a subring of $A[t]$. Of course
we can consider $\R$ as a subring of $\eR$ too. Both these embedding's of $\R$ would be useful for us.
 Set
\[
\eR_M = \bigoplus_{n \in \Z}I^nMt^n  \quad \text{and} \quad    \R_M = \bigoplus_{n \geq 0}I^n Mt^n.
\]
We call $\eR_M$
the extended Rees module of $M$ \wrt \ $I$ and we call $\R_M$ to be the Rees module of $M$ \wrt \ $I$.

\s Consider $L^I(M) = \bigoplus_{n \geq 0}M/I^{n+1}M$. We consider $L^I(M)$ as a $\eR$-module as follows:

Consider the exact sequence
\[
0 \xar \eR_M \xar M[t,t^{-1}] \xar L^I(M)(-1) \xar 0.
\]
Here $M[t,t^{-1}] = M \otimes_A A[t,t^{-1}]$. This exact sequence gives $L^I(M)$ a structure of $\eR$-module.
Since $\R$ is a subring of $\eR$; we also get that $L^I(M)$ is a $\R$-module. We may also see this directly
through the exact sequence
\[
0 \xar \R_M \xar M[t] \xar L^I(M)(-1) \xar 0
\]
\s \label{quasi-defn}\textbf{Quasi-finite modules}
It will be convenient at times to work a little more generally.  We extend definition of quasi-finite modules from that of  \cite[page 10]{HPV}.
Let $E = \bigoplus_{n\in \Z}E_n$ be a $\R$-module. We say $E$ is \emph{quasi-finite of order at least $s$} if
 \begin{enumerate}
   \item $E_n$ is a finitely generated $A$-module for all $n \in \Z$
   \item $E_n = 0$ for all $n \ll 0$.
   \item For $i = 0,\ldots, s-1$ we have $H^i_{\R_+}(E)_n = 0 $ for all $n \gg 0$.
 \end{enumerate}
\begin{remark}
Of course if $E$ is a finitely generated $\R$-module then it is quasi-finite of any order $s \geq 1$. In the next section we prove that if $M$ is \CM \ of dimension $r\geq 1$ and $I$ is an ideal of definition for $M$ then $L^I(M)$ is quasi-finite of order at-least  $r$.
\end{remark}

\s Let $E = \bigoplus_{n\in \Z}E_n$ be a non-necessarily finitely generated $\R$-module with $E_n = 0$ for all $n \ll 0$. An element $u \in \R_1$ is said to be $E$-\emph{filter regular } if
$(0 \colon_E u)_n = 0$ for all $n \gg 0$.

\begin{remark}
If $E$ is quasi-finite of order at-least $s (\geq 2) $ and $u$ is $E$-filter regular then $E/uE$ is quasi-finite of order at-least $s-1$.
This can be proved by noting that $(0\colon_E u)$ is  $\R_+$-torsion.
\end{remark}

\s Let $E = \bigoplus_{n\in \Z}E_n$ be a quasi-finite  $\R$-module of order at-least $s$. Let  $\bu = u_1,\ldots,u_r \in \R_1$ be a sequence and assume $r \leq s$.
We say $\bu$ is a $E$-filter regular sequence if $u_1$ is $E$-filter regular, $u_2$ is $E/u_1E$-filter-regular, \ldots, $u_r$ is $E/(u_1,\ldots,u_{r-1})E$ filter-regular.

\begin{proposition}
Assume that the residue field of $A$ is uncountable. Let $E$ be a quasi-finite $\R$-module of order at least $s$. Then there exists
$\bu = u_1,\ldots,u_s \in \R_1$ which is $E$-filter   regular  sequence.
\end{proposition}
\begin{proof}
It is sufficient to do this for $s =1$. In this case the result follows from \cite[2.7]{HPV}
\end{proof}
\begin{remark} Assume $M$ is \CM.
Let $\bx = x_1,\ldots,x_r$ be a $M$-superficial sequence \wrt \ $I$.  Set $u_i = x_it \in \R_1$ for $i = 1,\ldots,r$. In the next section we show that
$\bu = u_1,\ldots,u_r$ is a $L^I(M)$ filter-regular sequence. We do not need the residue field of $A$ to be uncountable.
\end{remark}
\section{$L^I(M)$}
\s\label{setup-2} \textbf{Setup and Introduction:} In this section  $M$ is a\emph{ \CM }\ $A$-module of dimension $r \geq 1$ and  $I$ is an ideal of definition for $M$. We consider the $\eR$-module
$L^I(M) = \bigoplus_{n\geq 0}M/I^{n+1}M$.
We prove that $L^I(M)$ is a quasi-finite $\R$-module of order at least $r$.  Let $\bx = x_1,\ldots,x_r$ be a $M$-superficial sequence \wrt \ $I$.  Set $u_i = x_it \in \R_1$ for $i = 1,\ldots,r$. We also show that
$\bu = u_1,\ldots,u_r$ is a $L^I(M)$ filter-regular sequence.

\s\label{indep} If $E$ is a graded $\eR$-module then notice that
\[
H^i_{\R_+}(E) \cong H^i_{\eR_+}(E) \quad \text{as $\R$-modules}.
\]
Note that $\eR_+$ denotes the ideal $\R_+ \eR$ of $\eR$.
The following result is known when $M = A$; see \cite[3.8]{Blancafort}.
\begin{lemma}\label{eRlc}
[with hypotheses as in \ref{setup-2}] As $\R$-modules:
\begin{enumerate}[\rm (1)]
  \item $H^1_{\eR_+}(\eR_M)$ is a quotient of $H^1_{\R_+}(\R_M)$.
  \item $H^i_{\eR_+}(\eR_M)  \cong H^i_{\R_+}(\R_M)$ for $i \geq 2$.
\end{enumerate}
\end{lemma}
\begin{proof}(Sketch)
We use \ref{indep} and the following short exact sequence of $\R$-modules
\[
0\xar \R_M \xar \eR_M \xar \eR_M/\R_M \xar 0.
\]
Notice $\eR_M/\R_M$ is $\R_+$-torsion.
\end{proof}

\begin{proposition}
$L^I(M)$ is quasi-finite of order $r = \dim M$.
\end{proposition}
\begin{proof}
Set $L = L^I(M)$.  Notice $H^i_{\R_+}(L) = H^i_{\eR_+}(L)$ as $\R$-modules.  Let $\bx = x_1,\ldots,x_r$ be a $M$-superficial sequence \wrt \ $I$.  Set $u_i = x_it \in \R_1$ for $i = 1,\ldots,r$.

Let $\bx = x_1,\ldots,x_r$ be a $M$-superficial sequence \wrt \ $I$.  Set $u_i = x_it \in \R_1$ for $i = 1,\ldots,r$.
It can be easily checked that $\bu$ is a $M[t,t^{-1}]$ regular sequence. So $H^i_{\eR_+}(M[t, t^{-1}]) = 0$ for $i = 0,\ldots, r-1$.

 We consider the exact sequence
\[
0 \xar \eR_M \xar M[t,t^{-1}] \xar L(-1) \xar 0.
\]
Taking local cohomology \wrt \ $\eR_+$ we get that

(a)  $H^i_{\eR_+}(L(-1)) \cong H^{i+1}_{\eR_+}(\eR_M)$ for $i = 0,\ldots, r-2$. \\
(b)  $H^{r-1}_{\eR_+}(L(-1))$ is a submodule of $H^r_{\eR_+}(\eR_M)$.

The result now follows from Lemma \ref{eRlc}, Remark \ref{indep} and \cite[15.1.5]{BSh}.
\end{proof}

\begin{proposition}\label{supTOfilter}
Let $\bx = x_1,\ldots,x_r$ be a $M$-superficial sequence \wrt \ $I$.  Set $u_i = x_it \in \R_1$ for $i = 1,\ldots,r$.
Then $\bu $ is a $L^I(M)$ filter-regular sequence.
\end{proposition}
\begin{proof}
Set $L = L^I(M)$.
We first show that $u_1$ is $L$ filter regular.
Notice
\[
(0 \colon_L u_1) = \bigoplus_{n \geq 0} \frac{I^{n+1}M \colon_M x_1}{I^nM}.
\]
Since $x_1$ is $M$-superficial it follows that $u_1$ is $L$ filter regular; see 1.1.

Check that
\[
\frac{L}{u_1 L} = \bigoplus_{n \geq 0}\frac{M}{x_1M + I^{n+1}M}  = L^I(M/x_1M).
\]
The result now follows from an easy induction on $\dim M$.
\end{proof}

\section{Koszul homology of quasi-finite modules \\ with respect to filter-regular sequence}
In this section we show some properties of Koszul homology of a quasi-finite module \wrt \ a filter regular sequence. We also compute the Koszul
homology of $L^I(M)$ \wrt \ $\bu = x_1t,\ldots,x_st$ where $x_1,\ldots,x_s$ is an $M$-superficial sequence \wrt\ $I$.
\begin{theorem}
Let $E$ be a quasi-finite $\R$-module of order at least $s$ and let $\bu = u_1,\ldots, u_s$ be a $E$-filter regular sequence. Then for
$i = 1,\ldots,s$ we have
\begin{enumerate}[\rm (1)]
  \item  $H_i(\bu, E)$ is a finitely generated $\R$-module. It is also $\R_+$-torsion. In particular $H_i(\bu, E)$ is a finitely generated
  $A$-module.
  \item If $\bu$ is $E$-regular sequence then $H_i(\bu, E) = 0$ for $i = 1,\ldots, s$.
  \item If $H_1(\bu, E) = 0$ then $\bu$ is a $E$-regular sequence.
\end{enumerate}
\end{theorem}
\begin{proof}
(1) We prove it by induction on $s$.

The case $s =1$.\\
Notice $H_1(u_1,E) = (0 \colon_{u_1} E)$. Since $u_1$ is $E$-filter regular we get that $H_1(u_1,E)$ is a finitely generated $A$-module and hence a finitely generated $\R$-module.  Clearly it is also    $\R_+$ torsion.

We assume the result for $s = r$ and prove for $s = r+1$. Let $\bu = u_1,\ldots,u_r,u_{r+1}$ and $\bu^\pr = u_1,\ldots, u_{r}$. We have
for all $i \geq 0$ an exact sequence
\begin{equation}\label{utou}
0 \xar H_0(u_{r+1}, H_i(\bu^\pr, E)) \xar H_i(\bu, E) \xar H_1(u_{r+1}, H_{i-1}(\bu^\pr, E)) \xar  0
\end{equation}
Using induction hypothesis it follows that for $i \geq 2$ the modules $H_i(\bu, E)$ are finitely generated $\R$-modules and also $\R_+$-torsion.
For $i = 1$ notice that

(a) $H_0(u_{r+1}, H_1(\bu^\pr, E))$ is finitely generated $\R$-module.
  It is also $\R_+$-torsion.

  (b) $H_1(u_{r+1}, H_0(\bu^\pr, E)) = H_1(u_{r+1},  E/\bu^\pr E)$. Since $u_{r+1}$ is $E/\bu^\pr E$-filter regular then by $s =1$ case we have
  that $H_1(u_{r+1}, H_0(\bu^\pr, E))$ is a finitely generated $\R$-module and it also $\R_+$-torsion

  The result follows.

  (2) The standard proof works.

  (3) Nothing to prove when $s = 1$. So assume $  s \geq 2$. Set $r = s -1$.  We use equation \ref{utou}.
   If $H_1(\bu, E) = 0$ then $H_0(u_{r+1}, H_1(\bu^\pr,E)) = 0$.
   So we have $H_1(\bu^\pr, E) = u_{r+1}H_1(\bu^\pr, E)$. Since $H_1(\bu^\pr, E)$ is a finitely generated graded $\R$-module and
   $u_{r+1}$ has positive degree it follows that $H_1(\bu^\pr, E) = 0$. By induction hypothesis it follows that
   $u_1,\ldots,u_r$ is a $E$-regular sequence.

  From \ref{utou} we also get $$H_1(u_{r+1}, H_0(\bu^\pr, E)) = H_1(u_{r+1},  E/\bu^\pr E) = 0.$$
   So $u_{r+1}$ is $E/\bu^\pr E$- regular. It follows that $\bu$ is a $E$-regular sequence.
\end{proof}

\begin{proposition}\label{compute}
Let $M$ be a \CM \ $A$-module of dimension $r \geq 1$ and let $I$ be an ideal of definition for $M$.
Let $\bx = x_1,\ldots,x_s$ be a $M$-superficial sequence \wrt \ $I$ with $s \leq r$.  Set $u_i = x_it \in \R_1$ for $i = 1,\ldots,s$.
Then $\bu $ is a $L^I(M)$ filter-regular sequence and
$$ H_1(\bu, L^I(M)) = \bigoplus_{n\geq 1} \frac{I^{n+1}M \cap \bx M}{\bx I^nM} = \Vc_I(\bx, M). $$
\end{proposition}
\begin{proof}
Set $L = L^I(M)$.  In \ref{supTOfilter} we have shown already that $\bu $ is a $L^I(M)$ filter-regular sequence.

Consider the exact sequence
\[
0 \xar \eR_M \xar M[t,t^{-1}] \xar L(-1) \xar 0.
\]
It can be easily checked that $\bu$ is a $M[t,t^{-1}]$ regular sequence. So $H_1(\bu, M[t,t^{-1}]) = 0$. Thus we have an exact sequence
\[
0 \xar H_1(\bu, L(-1)) \xar H_0(\bu, \eR_M ) \xar H_0(\bu, M[t,t^{-1}]  )  \xar H_0(\bu, L ) \xar 0.
\]
Notice
\[
H_0(\bu, \eR_M ) = \bigoplus_{n\in Z} \frac{I^{n}M}{\bx I^{n-1}M} \quad \text{and} \quad H_0(\bu, M[t,t^{-1}]  )  = M/\bx M[t,t^{-1}]
\]
So
\[
H_1(\bu, L(-1))  = \bigoplus_{n\in Z} \frac{I^{n}M \cap \bx M}{\bx I^{n-1}M}.
\]
The result follows.
\end{proof}
\section{local cohomology of quasi-finite modules $E$ \\ with
            $\ell(E_n)$ finite for all $n \in \Z$}
In this section we prove a suprising fact:  the local cohomology modules

\noindent $H^i_{\R_+}(L^I(M))$ are all *-Artinian for $i = 0,\ldots, \depth M -1$.
It is convenient to prove it in the generality of quasi-finite modules.
\s \label{setup-lc} Throughout this section $H^i(-) = H_{\R_+}^i(-)$ the $i$-th local cohomology functor \wrt\ $\R_+$. In this section we assume that
\begin{enumerate}
  \item $(A,\m)$ is complete with infinite residue field.
  \item $E$ is a quasi-finite module of order at least  $s$.
  \item There exists an $E$-filter regular sequence of length $s$.
  \item $\ell(E_n)$ finite for all $n \in \Z$.
\end{enumerate}

\begin{remark}
The hypothesis on existence of  $E$-filter regular sequence of length $s$ is automatically satisfied if $k$ is uncountable.
The assumption  "$\ell(E_n)$ finite for all $n \in \Z$" is to imitate that of $L^I(M)$. Finally if $M$ is\ CM \ and $A$ has infinite residue field then assumptions 2, 3, 4 are automatically satisfied for $L^I(M)$.
The assumption $A$ is complete is needed since we will use Matlis-Duality.
\end{remark}

\begin{theorem}\label{locArt}[with hypotheses as in \ref{setup-lc}]
For $i = 0,\ldots, s-1$ we have
\begin{enumerate}[\rm (1)]
  \item $\ell(H^i(E)_n) < \infty$ for all $n \in \Z$.
  \item $H^i(E)^\vee$ is a Noetherian $\R$-module.
   \item $H^i(E)$ is a *-Artinian $\R$-module.
  \end{enumerate}
\end{theorem}
\begin{proof}
We prove everything together by induction on $s$.

The case $s = 1$ \\
Clearly
 $\ell(H^0(E)_n) < \infty$ for all $n\in \Z$ and is zero for $n \ll 0$. By hypothesis $E$ is quasi-finite of order at least $1$. So   $H^0(E)_n = 0$ for all $n \gg 0$. The result follows.

We assume the result for $s = r$ and prove for $s = r+1$. Since $E$ is quasi-finite module of order at least  $r+1$ it is also  quasi-finite module of order at least  $r$. So by induction hypothesis applied to $E$ we have that
for $i = 0,\ldots,r-1$ the modules  $H^i(E)$ satisfy properties (1), (2) and (3).  It remains to prove that $H^r(E)$ satisfies properties (1), (2) and (3). \\
Let $u$ be $E$-filter regular. Set $F = E/uE$. We have an exact sequence
\[
0 \xar (0\colon_E u) \xar E(-1) \xrightarrow{u} E \xar F \xar 0.
\]
Since $(0\colon_E u)$ is $R_+$-torsion,  by using a standard trick, we get the exact sequence
\begin{align*}
0 \xar (0\colon_E u) \xar  H^0(E)(-1) &\xrightarrow{u}  H^0(E) \xar  H^0(F) \xar \\
                    H^1(E)(-1) &\xrightarrow{u} H^1(E) \xar  H^1(F) \xar \\
                    &\cdots                               \\
                    H^{r-1}(E)(-1) &\xrightarrow{u}  H^{r-1}(E) \xar  H^{r-1}(F) \xar\\
                    H^r(E)(-1) &\xrightarrow{u}  H^r(E).
\end{align*}
So we have an exact sequence
\begin{equation*}
H^{r-1}(F) \xrightarrow{\delta}    H^r(E)(-1) \xrightarrow{u}  H^r(E).  \tag{*}
\end{equation*}
 Since $F$ is quasi-finite of order at least $r$ we get that $H^{r-1}(F)$ satisfies  properties (1), (2) and (3).  We prove that $H^r(E)$ satisfies properties (1), (2) and (3).

 (1) By hypothesis on $E$ we have $H^r(E)_n = 0$ for all $n \gg 0$ say  from $n\geq c+1$. By equation (*) we have
 $H^{r-1}(F)_{c+1} \xrightarrow{\delta}    H^r(E)_c \xar  H^r(E)_{c+1} = 0$.
 Since  $H^{r-1}(F)$ satisfies (1) we get that  $H^r(E)_c$ has finite length. Once can induct on $j $ to show that  $H^r(E)_{c-j}$ has finite length for all $j \geq 0$.

 (2) We have an exact sequence of $\R$-modules
$$ H^r(E)^\vee \xrightarrow{u}  H^r(E)^\vee(+1) \xrightarrow{\delta^\vee}  H^{r-1}(F)^\vee. $$
Set $W =  H^r(E)^\vee$.  Since $H^{r-1}(F)^\vee $ is finitely generated $\R$-module it follows that $W/uW(+1)$ (and so $W/uW$) is finitely generated.

Say $V = <\xi_1,\ldots,\xi_m>$ is a $\R$-submodule of $W$ such that $W = V + u W$. We prove $W = V$. This we do degree-wise.
 By hypothesis on $E$ we have $H^r(E)_n = 0$ for all $n \gg 0$. So $W_n = 0$ for all $n \ll 0$ say from $n < c$.
 Since $\deg u = 1$ we have $W_c = V_c$. Notice
 \[
 W_{c+1} = V_{c+1} + uW_c = V_{c+1} + uV_c =  V_{c+1}.
 \]
 By induction on $j$  it is easy to show $W_{c+j} = V_{c+j}$ for all $j \geq 0$.

 (3) This follows from Matlis duality.
 \end{proof}

\begin{corollary}\label{loc-CA}[with hypotheses as in \ref{setup-lc}]
 For $i = 0,\ldots,s-1$ set $\af(E)_i = \ann_{\R} H^i(E)$ and $\q_i(E) = \af(E)_i \cap A$. If $H^i(E) \neq 0$
then $\q_i(E)$ is $\m$-primary.
\end{corollary}
\begin{proof}
Fix $i$ with $0 \leq i \leq s-1$.
Set $D_i = H^i(E)$ and assume it is non-zero. It is easily checked using Matlis duality that $\ann_{\R} D_i = \ann_{\R} D_i^\vee$.

Notice  $D_i^\vee$ is a finitely generated $\R$-module such that $\ell((D_i^\vee)_n)$ is finite for all $n$. Let
$m_1,\ldots,m_s$ be homogeneous generators of $D_i^\vee$. Consider the map
\begin{align*}
\frac{\R}{\af_i(E)} &\xrightarrow{\psi} \bigoplus_{j=1}^{s}D_i^\vee(-\deg m_j) \\
\ov{t} &\mapsto (tm_1,\ldots,tm_s).
\end{align*}
Clearly $\psi$ is injective.
 Taking degree zero part of this embedding gets us that
$\q_i(E)$ is $\m$-primary.
\end{proof}

\section{Proof of main theorem}
The proof of the following result is inspired by Theorem 8.1.2 from \cite{BH}; (also see \cite[Theorem 1]{Roberts}). However  we have to be extra careful at a few places.
The hypothesis of our result is not exactly similar and we are dealing with infinitely generated modules.

\begin{theorem}
Let $(A,\m)$ be  a complete Noetherian ring  with an infinite residue field  and let $I$ be an $\m$-primary ideal in $A$. Let
$N$ be a quasi-finite $\R$-module  of order at least  $m$.
 Assume $\bu = u_1,\ldots,u_m \in \R_1$ is  a $N$ filter-regular sequence such that
  \[
  H^*_{\bu}(N) = H^*_{\R_+}(N)
  \]
Also assume that $\ell(N_n)$ is finite for all $n\in \Z$.
Set $\bu^\pr = u_1,\ldots,u_n$ with $n \leq m$ and let
\[
\K = \K(\bu^\pr, N)  \colon 0 \rt E_n \rt \cdots\rt E_1\rt E_0\rt 0
\]
be the Koszul complex of $\bu^\pr$ with coefficients in $N$.

For  $j = 0,\ldots, m-1$
set $\bbf_j = \ann_{\R} H^j_{\R_+}(N) $ and $\q_j = A\cap \bbf_j$.
Then  $\q_0\q_1\cdots\q_{n-1}$ annihilates  $H_1(\K(\bu^\pr, N))$.
\end{theorem}
\begin{proof}
Let $\C$ be the Cech co-chain complex on $u_1,\ldots, u_m$. We shift $\C$ $m$-places and write it as a chain complex
\[
\D \colon 0 \rt D_m \rt \cdots \rt D_1 \rt D_0 \rt 0.
\]
By construction $H_i(N\otimes \D) = H^{m-i}_{\R_+}(N)$.

Consider the chain bicomplex $\X = \D \otimes \K$. We consider the two standard spectral sequences to compute the homology of  $\Y = \Tot(\X)$; the total
complex of $\X$.

\emph{The first spectral sequence:}\\
${}^{I}E^0_{pq} = D_p\otimes K_q$. So
\begin{align*}
{}^{I}E^1_{pq} &= H_q(D_p \otimes \K) \\
&= D_p\otimes H_q(\K), \quad \text{since $D_p $ is flat}.
\end{align*}
By Theorem 3.1 we have that $H_q(\K)$ is $\R_+$-torsion for all $q > 0$.
 It follows that
\[
{}^{I}E^1_{pq} = \begin{cases} 0 & \text{for $q > 0$ and $p \neq m$,} \\ H_q(\K) & \text{for $q > 0$ and $p = m$,}\\ D_p\otimes H_0(\K) &\text{for $q = 0$.} \end{cases}
\]
Therefore
\[
{}^{I}E^2_{pq} = \begin{cases} 0 & \text{for $q > 0$ and $p \neq m$,} \\ H_q(\K) & \text{for $q > 0$ and $p = m$,}\\ H^{m-p}_{\R_+}(H_0(\K)) &\text{for $q = 0.$} \end{cases}
\]
Observe that this spectral sequence \emph{collapses }at ${}^{I}E^2$. So $H_{m+i}(\Y) \cong H_i(\K)$ for $1\leq i \leq n$.

\emph{The second spectral sequence:}\\
${}^{II}E^0_{pq} = D_q\otimes K_p$. So
$$ {}^{II}E^1_{pq} = H_q(\D\otimes K_p) = H^{m-q}_{R_+}(K_p) = \left( H^{m-q}_{R_+}(N) \right)^{\binom{n}{p}}.                              $$
By construction $\q_{m-q}$ annihilates ${}^{II}E^1_{pq}$ if $q \neq 0$. Since  ${}^{II}E^\infty_{pq}$ is a subquotient of ${}^{II}E^1_{pq}$ we get that
  $\q_{m-q}$ annihilates $E^\infty_{pq}$ if $q \neq 0$.

  Let $0 = V_{-1} \subseteq V_0 \subseteq V_1 \subseteq \cdots \subseteq V_{j-1} \subseteq V_j = H_{m+1}(\Y)$ be the filtration
  such that  ${}^{II}E^\infty_{p,m+1-p} \cong V_p/V_{p-1}$.
  Notice  ${}^{II}E^\infty_{p,m+1-p} = 0$ for $p > n$ and $m+1 -p >m$ (equivalently $p < 1$). So in the filtration $1 \leq p \leq n$.
  Notice in this range $q = m+1-p \neq 0$ (otherwise $p = m+1 > n$). So $\q_{m-q} = \q_{p-1}$ annihilates  ${}^{II}E^\infty_{p,m+1-p} $ for
  the range $1 \leq p \leq n$. It follows that $\q_0\q_1\cdots\q_{n-1}$ annihilates $H_{m+1}(\Y)$. The result follows since $H_{m+1}(\Y)= H_1(\K)$.
\end{proof}

\begin{theorem}\label{man}
Let $(A,\m)$ be a complete \CM \ local ring of dimension  with infinite residue field and dimension $d \geq 1$. Let $I$ be an $\m$-primary ideal in $A$. Set $L = L^I(A)$. For $i = 0,\ldots, d-1$ set
$\q_i = A \cap \ann_{\R}H^i_{R_+}(L)$. For $r =1,\ldots,d$ set
\[
\af_r(I) = \bigcap_{\substack{\text{$\bx = x_1,\ldots,x_r$ is a } \\ \text{ superficial sequence of $I$}}} \ann_A \Vc_I(\bx)
\]
Then $\af_r(I) \supseteq \q_0\cdots\q_{r-1}$. In particular if $\depth G_I(A) < r$ then $\af_r(I)$ is $\m$-primary.
\end{theorem}
\begin{proof}
By 2.4, $L$ is quasi-finite $\R$-module of order at least $d$.
Fix $r \geq 1$. Let $\bx^\pr = x_1,\ldots,x_r$ be an $I$-superficial sequence.
 Then $\bx^\pr$ can be extended to a maximal superficial sequence $\bx = x_1,\ldots,x_r,x_{r+1},\ldots,x_d$.
  Set $u_i = x_it \in \R_1$. Then by 2.5
$\bu = u_1,\ldots,u_d$ is a $L$-filter regular sequence. Since $(\bx)$ is a reduction of $I$ it follows that
$\bu$ generates $\R_+$ up to radical. So $H^{i}_{\bu}(L) = H^i_{\R_+}(L)$. Set $\bu^\pr = u_1,\ldots,u_r$. Let $\K(\bu^\pr,L)$ be the Koszul complex on
$\bu^\pr$ with coefficients in $L$. By 3.2 we get that $H_1(\bu^\pr, L) = \Vc_I(\bx^\pr)$. From Theorem 5.1.
we get $\ann_A \Vc_I(\bx^\pr) \supseteq \q_0\cdots\q_{r-1}$. Since $\bx^\pr$ was an arbitary superficial sequence of length $r$ we get $\af_r(I) \supseteq \q_0\cdots\q_{r-1}$.
\end{proof}

We now drop the assumption that $A$ is complete.
\begin{theorem}\label{main}
Let $(A,\m)$ be a \CM \ local ring with infinite residue field and dimension $d \geq 1$. Let $I$ be an $\m$-primary ideal and let $1 \leq r \leq d$. Then
\[
\af_r(I\widehat{A}) \cap A \subseteq \af_r(I).
\]
Furthermore
if $\depth G_I(A) < r$ then
$\af_r(I)$ is $\m$-primary.
\end{theorem}
\begin{proof}
Let $\widehat{A}$ be the completion of $A$. Let $\bx = x_1,\ldots,x_r$ be an $I$-superficial sequence. Then $\bx$ considered as a sequence in $\widehat{A}$ is also a $\widehat{I}$-superficial sequence. Furthermore $\Vc_{I\widehat{A}}(\bx) = \Vc_{I}(\bx)$ since it is of finite length.
It follows that $\ann_{\widehat{A}} \Vc_{I\widehat{A}}(\bx) \cap A = \ann_A \Vc_I(\bx)$.

Notice
\[
\af_r(I\widehat{A}) \subseteq \bigcap_{\substack{\text{$\bx = x_1,\ldots,x_r$ is a } \\ \text{ superficial sequence of $I$}}} \ann_A \Vc_{I\widehat{A}}(\bx).
\]
Therefore $\af_r(I\widehat{A}) \cap A \subseteq \af_r(I)$. Furthermore as $G_{I\widehat{A}}(\widehat{A}) = G_I(A)$ has depth $< r$ we have that
$\af_r(I\widehat{A})$ is $\widehat{\m}$-primary. It follows that $\af_r(I)$ is $\m$-primary.
\end{proof}

\section{Powers of $I$}
In this section we invesitigate $\af_r(I^l)$ for $ l \geq 1$.
One of the advantages of $L^I(A)$ is that $L^I(A)(-1)$ commutes with the Veronese functor.
Clearly
\[
\left(L^I(A)\right)^{<l>}  = L^{I^l}(A)(-1) \quad \text{for each} \ l \geq 1.
\]
Also note that for the Rees algebras we have
\[
\Rc(I^l) = \R^{<l>} \quad \text{and} \quad \Rc(I^l)_+ = \R_+^{<l>}.
\]
Local cohomology also commutes with the Veronese functor. So we have that
\[
H^i_{\Rc(I^l)_+} \left(L^{I^l}(A)(-1)\right)  \cong  \left(H^i_{\R_+}(L^I(A))(-1)\right)^{<l>} \quad \text{for all}  \ l \geq 1.
\]

We first prove the following general result.
\begin{lemma}\label{ml-powers}
Let $(A,\m)$ be a Noetherian local ring and let $I$ be an $\m$-primary ideal. 
Let $E$ be a finitely generated graded $\R$-module with $\ell(E_n) < \infty$ for all $n \in \Z$.
For $l \geq 1$ set
\[
\q(I^l)_E = \left(\ann_{\Rc(I^l)} E^{<l>} \right)\cap A.
\]
Then
\begin{enumerate}[\rm(1)]
\item
$\q(I^l)_E$ is $\m$-primary for each $l \geq 1$.
\item
For each $r,l \geq 1$ we have
\[
\q(I^l)_E \subseteq   \q(I^{rl})_E.
\]
\item
The set
\[
\mathcal{C} = \{ \q(I^l)_E \mid l \geq 1 \},
\]
has a unique maximal element which we denote as $\q(I)^\infty_E$.
\end{enumerate}
\end{lemma}
\begin{proof}
$\rm(1)$. Fix $l \geq 1$. Then $E^{<l>}$ is a finitely generated graded $\Rc(I^l)$-module with $\ell(E^{<l>}_j)$ finite for all $j \in \Z$. So by an argument similar to Corollary 4.4 we have that $\q(I^l)_E$ is $\m$-primary.

$\rm(2)$. Notice
\[
\left(E^{<l>}\right)^{<r>} = E^{<rl>}.
\]
Thus it suffices to prove the result for $l=1$.
Let $a \in \q(I)_E$. Then $a E_j = 0$ for all $j \in \Z$. So we have that $a \in \ann_{\Rc(I^r)} E^{<r>}$. Also as 
$a \in A$ we have that $a \in \q(I^r)_E$.

$\rm(3)$ Suppose $\q(I^l)_E$ and $\q(I^r)_E$ are maximal elements in $\mathcal{C}$. By $\rm(2)$ we have that
\[
\q(I^l)_E \subseteq   \q(I^{rl})_E \quad \text{and} \quad
\q(I^r)_E \subseteq   \q(I^{rl})_E.
\]
By maximality of $\q(I^l)_E$ in $\mathcal{C}$ we have that $\q(I^l)_E =  \q(I^{rl})_E$. Similarly
$\q(I^r)_E =  \q(I^{rl})_E$. So  $\q(I^l)_E = \q(I^r)_E$.
\end{proof}
\begin{question}(with hypotheses as above)
Is  
\[
\q(I)^\infty_E= \q(I^l)_E  \quad  \text{for all} \ l \gg 0?
\]
\end{question}

We now prove the following result:
\begin{theorem}\label{main-2}
Let $(A,\m)$ be a \CM \ local ring with infinite residue field and dimension $d \geq 1$. Let $I$ be an $\m$-primary ideal and let $1 \leq r \leq d$. 
If $\depth G_I(A) < r$ then
$$\bigcap_{n\geq 1}\af_r(I^n) \quad \text{ is $\m$-primary}.$$
\end{theorem}
\begin{proof}
By Theorem \ref{main}
\[
\af_r(I\widehat{A}) \cap A \subseteq \af_r(I).
\]
Thus $\af_r(I^n\widehat{A}) \cap A \subseteq \af_r(I^n)$ for all $n \geq 1$. Thus it suffices to prove the result when $A$ is complete.
Let $l \geq 1$. For $i = 0,1,\ldots, r-1$, define
$$\q_i(I^l) = \left(\ann_{\Rc(I^l)} H^i_{\Rc(I)_+} (L^{I^l}(A)) \right) \cap A.$$
By Theorem \ref{man}
\[
\af_r(I^l) \supseteq \q_0(I^l)\q_1(I^l)\cdots\q_{r-1}(I^l).
\]
 For $i = 0,1,\ldots, r-1$ set
 $$D_i(l) = H^i_{\Rc(I)_+}\left(L^{I^l}(A)(-1)\right)^\vee.$$
  Note
 that by Matlis duality 
 $$D^i(l)^\vee = H^i_{\Rc(I)_+}\left(L^{I^l}(A)(-1)\right).$$ 
 Clearly
 \[
 \q_i(I^l) = \left(\ann_{\Rc(I^l)} D_i(l)\right)\cap A \quad \text{for} \ i = 0,1,\ldots, r-1.
 \]
 Since $L^I(A)$ and local cohomology behaves well with respect to the Veronese functor we have
 that for all $l \geq 1$ we have
 \[
D_i(l) = D_i(1)^{<l>}  \quad \text{for} \ i = 0,1,\ldots, r-1.
 \]
 By Lemma \ref{ml-powers}(2) we have $\q_i(I^l) \supseteq \q_i(I)$ for all $l \geq 1$ and for all $i = 0,\ldots,r-1$.
Therefore we have 
\[
\af_r(I^l) \supseteq \q_0(I)\q_1(I)\cdots\q_{r-1}(I) \ \text{for all} \ l \geq 1.
\]
It follows that $\bigcap_{n\geq 1}\af_r(I^n) \quad \text{ is $\m$-primary}.$
\end{proof}
We end our paper with the following:
\begin{question}(with hypothesis as above)
Is $\af_r(I^n)$ constant for all $n \gg 0$?
\end{question}
\providecommand{\bysame}{\leavevmode\hbox to3em{\hrulefill}\thinspace}
\providecommand{\MR}{\relax\ifhmode\unskip\space\fi MR }
\providecommand{\MRhref}[2]{%
  \href{http://www.ams.org/mathscinet-getitem?mr=#1}{#2}
}
\providecommand{\href}[2]{#2}

%\bibliographystyle{amsplain}
%\bibliography{ref}
\end{document}